\documentclass{amsart}

\usepackage{amsmath,amscd}
\usepackage{amssymb}

\usepackage[all]{xy}

\usepackage{times}

\usepackage[T1]{fontenc}
\usepackage{graphicx}
\usepackage[svgnames]{xcolor}
\usepackage{framed}

\usepackage{tikz}

\usepackage{url}

\author{Liran Shaul}
\address{Universiteit Antwerpen, Departement Wiskunde-Informatica, Middelheim campus,
Middelheimlaan 1,
2020 Antwerp, Belgium}
\email{Liran.Shaul@uantwerpen.be}
\newtheorem{thm}[equation]{Theorem}
\newtheorem{cor}[equation]{Corollary}
\newtheorem{prop}[equation]{Proposition}
\newtheorem{lem}[equation]{Lemma}
\theoremstyle{definition}

\newtheorem{rem}[equation]{Remark}
\newtheorem{exa}[equation]{Example}

\newtheorem{warn}[equation]{Warning}

\newcommand{\opn}{\operatorname}

\newcommand{\mfrak}[1]{\mathfrak{#1}}

\newcommand{\mrm}[1]{\mathrm{#1}}

\renewcommand{\k}{\Bbbk}

\renewcommand{\a}{\mfrak{a}}
\renewcommand{\b}{\mfrak{b}}

\newcommand*\openquote{\makebox(25,-10){\scalebox{3}{``}}}
\newcommand*\closequote{\makebox(25,-10){\scalebox{3}{''}}}
\colorlet{shadecolor}{White}

\makeatletter
\newif\if@right
\def\shadequote{\@righttrue\shadequote@i}
\def\shadequote@i{\begin{snugshade}\begin{quote}\openquote}
\def\endshadequote{%
  \if@right\hfill\fi\closequote\end{quote}\end{snugshade}}
\@namedef{shadequote*}{\@rightfalse\shadequote@i}
\@namedef{endshadequote*}{\endshadequote}
\makeatother

\numberwithin{equation}{section} 

\begin{document}

\title{Hochschild cohomology commutes with adic completion}
\begin{abstract}
For a flat commutative $\k$-algebra $A$ such that the enveloping algebra $A\otimes_{\k} A$ is noetherian, given a finitely generated bimodule $M$, we show that the adic completion of the Hochschild cohomology module $\mrm{HH}^n(A/\k,M)$ is naturally isomorphic to $\mrm{HH}^n(\widehat{A}/\k,\widehat{M})$. To show this, we (1) make a detailed study of derived completion as a functor $\mrm{D}(\opn{Mod} A) \to \mrm{D}(\opn{Mod} \widehat{A})$ over a non-noetherian ring $A$; (2) prove a flat base change result for weakly proregular ideals; and (3) Prove that Hochschild cohomology and analytic Hochschild cohomology of complete noetherian local rings are isomorphic, answering a question of Buchweitz and Flenner. Our results makes it possible for the first time to compute the Hochschild cohomology of $\k[[t_1,\dots,t_n]]$ over any noetherian ring $\k$, and open the door for a theory of Hochschild cohomology over formal schemes.
\end{abstract}

\thanks{The author acknowledges the support of the European Union for the ERC grant No 257004-HHNcdMir.}

\maketitle

\setcounter{tocdepth}{1}
\tableofcontents

All rings in this paper are assumed to be commutative and unital. 

\setcounter{section}{-1}
\section{Introduction} 

Since its introduction in \cite{Ho}, Hochschild cohomology has been the prominent cohomology theory for associative algebras. In commutative algebra and algebraic geometry, its importance was first demonstrated by the celebrated theorem of Hochschild, Kostant and Rosenberg. See \cite{Io} for a survey of the use of Hochschild cohomology in commutative algebra. 

The aim of this paper is to initiate the study of Hochschild cohomology in the category of  adic rings. Adic rings, the affine pieces of the theory of formal schemes, are by definition commutative noetherian rings $A$ which are $\a$-adically complete with respect to some ideal $\a\subseteq A$. In the survey \cite{Io} mentioned above, Ionescu states:

\begin{shadequote}
In our survey no results about Hochschild cohomology of a topological algebra
were mentioned. This is because these kind of results are missing completely.
\end{shadequote}
As far as we know, little has changed regarding this statement since then. One of the main difficulties in developing such a theory can be already observed in the most simple example of an adic ring: Let $\k$ be a field of characteristic $0$, and let $A = \k[[t]]$. The construction of Hochschild cohomology involves the enveloping algebra of $A$. But even in this simple case, the enveloping algebra $\k[[t]]\otimes_{\k} \k[[t]]$ is a non-noetherian ring of infinite Krull dimension, so it is very difficult to do homological algebra over it. Passing to the completion of this enveloping algebra, one obtains the much more manageable completed tensor product $\k[[t]]\widehat{\otimes}_{\k} \k[[t]] \cong \k[[t_1,t_2]]$. However, from a homological point of view, this step is highly non-trivial, as it involves the ring map from $\k[[t]]\otimes_{\k} \k[[t]]$ to its completion. Completions of non-noetherian rings are in general poorly behaved (for instance, they need not be flat). In this paper we develop the homological tools needed to overcome this difficulty, and use them to study the Hochschild cohomology of such adic algebras.

Here is a more detailed description of the content of this paper. First, in Section \ref{section:review} we review some preliminaries on Hochschild cohomology and about the derived torsion and derived completion functors. In particular, we recall the notion of a weakly proregular ideal in a commutative ring. Weak proregularity is the right condition in order for the derived torsion and derived completion functors to possess good behavior.

In Section \ref{section:WPR}, we prove that in most enveloping algebras of adic rings occurring in nature, the ideal of definition of the adic topology is weakly proregular. This result has also interesting implications in derived algebraic geometry of formal schemes. See Corollary \ref{cor:derived-fiber-of-formal}.

Given a (not necessarily noetherian) ring $A$ and a weakly proregular ideal $\a\subseteq A$, we study in Section \ref{section:widehat} the derived functors of the functors
\[
\widehat{\Gamma}_{\a} (M) := \varinjlim \opn{Hom}_A(A/{\a}^n,M) : \opn{Mod} A \to \opn{Mod} \widehat{A}
\]
and
\[
\widehat{\Lambda}_{\a} (M) := \varprojlim A/{\a}^n \otimes_A M : \opn{Mod} A \to \opn{Mod} \widehat{A},
\]
where $\widehat{A}$ is the $\a$-adic completion of $A$. To cope with the possible lack of flatness of the completion map $A\to\widehat{A}$, we use DG-homological algebra techniques. From the results of this section, we deduce in Corollary \ref{cor:generalized-GM} the following generalized Greenlees-May duality:
\begin{thm}
Let $A$ be a commutative ring, let $\a\subseteq A$ be a weakly proregular ideal, and let $M,N \in \mrm{D}(\opn{Mod} A)$. Then there are isomorphisms
\[
\mrm{L}\widehat{\Lambda}_{\a} (\mrm{R}\opn{Hom}_A(M,N)) \cong \mrm{R}\opn{Hom}_A( \mrm{R}\widehat{\Gamma}_{\a} (M), N) \cong \mrm{R}\opn{Hom}_A(M,\mrm{L}\widehat{\Lambda}_{\a} (N))
\]
of functors 
\[
\mrm{D}(\opn{Mod} A) \times \mrm{D}(\opn{Mod} A) \to \mrm{D}(\opn{Mod} \widehat{A}).
\]
\end{thm}

Using the results of Sections \ref{section:WPR} and \ref{section:widehat}, the main results of this paper are obtained in Section \ref{section:main}. First, in Theorem \ref{thm:LLambda-Of-Hochscild} we provide formulas which describe the effect of applying the derived completion functor to the Hochschild cohomology complex of a not necessarily adic algebra. The next major result, repeated in Corollary \ref{cor:adic-hh-is-hh} below, reduces the problem of computing the Hochschild cohomology of an adic algebra to a problem over noetherian rings:
\newpage
\begin{thm}
Let $\k$ be a commutative ring, and let $A$ be a flat noetherian $\k$-algebra. Assume $\a\subseteq A$ is an ideal, such that $A$ is $\a$-adically complete, and such that $A/\a$ is essentially of finite type over $\k$. Let $I:= \a\otimes_{\k} A + A\otimes_{\k} \a$, and set $A\widehat{\otimes}_{\k} A := \Lambda_I(A\otimes_{\k} A)$. Then for any $M \in \opn{Mod} A\otimes_{\k} A$ which is $I$-adically complete (for example, any $\a$-adically complete $A$-module, or more particularly, any finitely generated $A$-module), there is a functorial isomorphism
\[
\mrm{R}\opn{Hom}_{A\otimes_{\k} A} (A,M) \cong \mrm{R}\opn{Hom}_{A\widehat{\otimes}_{\k} A}(A,M)
\]
in $\mrm{D}(\opn{Mod} A)$, and the ring $A\widehat{\otimes}_{\k} A$ is noetherian.
\end{thm}
Using the results of the paper \cite{BF}, as a corollary of this result, we are able to prove in Corollary \ref{cor:Answer-To_BF} that Hochschild cohomology and analytic Hochschild cohomology of complete noetherian local algebras coincide, answering a question of Buchweitz and Flenner. Finally, Section \ref{section:main} ends with Theorem \ref{thm:main} which proves the result mentioned in the title of the paper. More precisely, we show: 
\begin{thm}
Let $\k$ be a commutative ring, and let $A$ be a flat noetherian $\k$-algebra such that $A\otimes_{\k} A$ is noetherian. Let $\a\subseteq A$ be an ideal, and let $M$ be a finitely generated $A\otimes_{\k} A$-module.
Then for any $n \in \mathbb{N}$, there is a functorial isomorphism
\[
\Lambda_{\a} \left(\opn{Ext}^n_{A\otimes_{\k} A}(A,M)\right) \cong \opn{Ext}^n_{\widehat{A}\otimes_{\k} \widehat{A}}(\widehat{A},\widehat{M})
\]
If moreover, either
\begin{enumerate}
\item $\k$ is a field, or 
\item $A$ is projective over $\k$, $\a$ is a maximal ideal, and $M$ is a finitely generated $A$-module,
\end{enumerate}
there is also a functorial isomorphism
\[
\Lambda_{\a} \left( \mrm{HH}^n(A/\k,M )  \right) \cong \mrm{HH}^n(\widehat{A}/\k,\widehat{M}).
\]
\end{thm}

In the short and final Section \ref{section:hhomology}, we briefly discuss analog results for Hochschild homology.

\begin{warn}
Contrary to the convention in many papers in the field, unless stated otherwise, we do not assume that rings are noetherian.
\end{warn}

\section{Preliminaries on completion, torsion and Hochschild cohomology}\label{section:review}

Given a commutative ring $A$, we denote by $\opn{Mod} A$ the abelian category of $A$-modules, and by $\mrm{D}(\opn{Mod} A)$ its (unbounded) derived category. If $A$ is noetherian, we will denote by $\mrm{D}_{\mrm{f}}(\opn{Mod} A)$ the triangulated subcategory made of complexes with finitely generated cohomologies. We will freely use resolutions of unbounded complexes, following \cite{Sp}.

\subsection{Completion and torsion}

References for the material in this section are \cite{AJL1,AJL2,GM,PSY1,PSY2,Sc,Si,Ye1}. See \cite[Remark 7.14]{PSY1} for a brief discussion on the history of this material. Let $A$ be a commutative ring, and let $\a\subseteq A$ be a finitely generated ideal. The $\a$-torsion functor $\Gamma_{\a} (-) : \opn{Mod} A \to \opn{Mod} A$ is defined by
\[
\Gamma_{\a} (M) := \varinjlim \opn{Hom}_A(A/{\a}^n,M).
\]
This functor is a left exact additive functor. We denote its (total) right derived functor by $\mrm{R}\Gamma_{\a}:\mrm{D}(\opn{Mod} A) \to \mrm{D}(\opn{Mod} A)$. It is computed using K-injective resolutions. See the book \cite{BS} for a detailed study of the $\a$-torsion functor and its derived functor in the noetherian case. More important in this paper is the $\a$-adic completion functor. This functor $\Lambda_{\a}(-) :\opn{Mod} A \to \opn{Mod} A$ is defined by
\[
\Lambda_{\a} (M) := \varprojlim A/{\a}^n \otimes_A M.
\]
This functor is additive, but in general is neither left exact nor right exact (even when $A$ is noetherian, see \cite[Example 3.20]{Ye1}). It does however preserve surjections. We denote by $\mrm{L}\Lambda_{\a}:\mrm{D}(\opn{Mod} A) \to \mrm{D}(\opn{Mod} A)$ its left derived functor. By \cite[Section 1]{AJL1}, it can be computed using K-flat resolutions. Both of the functors $\Gamma_{\a}(-), \Lambda_{\a}(-)$ are idempotent (\cite[Corollary 3.6]{Ye1}). 

For any ring $A$, the $A$-module $\Lambda_{\a}(A)$ has the structure of a commutative $A$-algebra, called the completion of $A$. If $A$ is noetherian then $\Lambda_{\a}(A)$ is flat over $A$, but if $A$ is not noetherian this does not always holds. For example, if $A$ is any countable ring which is not coherent, then the completion map $A[x] \to A[[x]]$ is not flat (\cite[Tag 0AL8]{Stack}). The ring $\Lambda_{\a}(A)$ is noetherian if and only if the ring $A/\a$ is noetherian (\cite[Tag 05GH]{Stack}). If $A$ is noetherian and $M$ is a finitely generated $A$-module, then there is an isomorphism of functors $\Lambda_{\a} (M) \cong \Lambda_{\a}(A) \otimes_A M$, so in particular in that case, $\Lambda_{\a}(-)$ is exact on the category of finitely generated $A$-modules (\cite[Tag 00MB]{Stack}). We will sometime denote by $\widehat{A}$ (respectively $\widehat{M}$) the $A$-algebra (resp. $A$-module) $\Lambda_{\a}(A)$ (resp. $\Lambda_{\a}(M)$). For any ring $A$, the $A$-modules $\Gamma_{\a}(M)$ and $\Lambda_{\a}(M)$ carry naturally the structure of $\widehat{A}$-modules, and so one may view the $\a$-torsion and $\a$-completion functors as functors $\opn{Mod} A \to \opn{Mod} \widehat{A}$. Section \ref{section:widehat} below is dedicated to a study of the functors obtained from this observation.

Given a ring $A$, and an element $\a\in A$, the infinite dual Koszul complex associated to it is the complex
\[
\opn{K}^{\vee}_{\infty}(A; (a)) := 
\bigl( \cdots \to 0 \to A \to A[{a}^{-1}] \to 0 \to
\cdots \bigr) 
\]
concentrated in degrees $0,1$. If $(a_1,\dots,a_n)$ is a finite sequence of elements in $A$, then the infinite dual Koszul complex associated to it is the complex
\[
\opn{K}^{\vee}_{\infty}(A; (a_1,\dots,a_n)) := \opn{K}^{\vee}_{\infty}(A; (a_1)) \otimes_A \dots \otimes_A \opn{K}^{\vee}_{\infty}(A; (a_n)).
\]
It is a bounded complex of flat $A$-modules. Given an ideal $\a\subseteq A$, and a finite sequence $\mathbf{a}$ of elements of $A$ that generate $\a$, by \cite[Corollary 4.26]{PSY1}, there is a morphism of functors
\[
\mrm{R}\Gamma_{\a} (-) \to \opn{K}^{\vee}_{\infty}(A; \mathbf{a}) \otimes_A -.
\]
The sequence $\mathbf{a}$ is called weakly proregular if this morphism is an isomorphism of functors. This notion is actually independent of $\mathbf{a}$, and depends only on the ideal $\a$ generated by it (\cite[Lemma 3.3]{Sc}). Hence, we say a finitely generated ideal $\a$ is weakly proregular if some (or equivalently, any) finite sequence that generates it is weakly proregular. In a noetherian ring, any ideal and any finite sequence are weakly proregular, but there are examples of finitely generated (even principal) ideals in non-noetherian rings which are not weakly proregular.

Given a ring $A$ and a finite sequence $\mathbf{a}$ of elements of $A$, the infinite dual Koszul complex has an explicit free resolution, called the telescope complex, and denoted by $\opn{Tel}(A;\mathbf{a})$. This resolution is a bounded complex of countably generated free $A$-modules (\cite[Lemma 5.7]{PSY1}). In particular, if the ideal $\a$ generated by $\mathbf{a}$ is weakly proregular, then there is also an isomorphism of functors
\[
\mrm{R}\Gamma_{\a} (-) \cong\opn{Tel}(A; \mathbf{a}) \otimes_A -.
\]
Moreover, in this case, by \cite[Corollary 5.25]{PSY1}, there is also an isomorphism of functors
\[
\mrm{L}\Lambda_{\a} (-) \cong \opn{Hom}_A(\opn{Tel}(A;\mathbf{a}), -).
\]
It follows that if $A$ is a commutative ring, and $\a$ is a weakly proregular ideal, then both of the functors 
$\mrm{R}\Gamma_{\a}, \mrm{L}\Lambda_{\a}$ have finite cohomological dimension, and there is a bifunctorial isomorphism, the Greenlees-May duality:
\[
\mrm{R}\opn{Hom}_A(\mrm{R}\Gamma_{\a}(M), N) \cong \mrm{R}\opn{Hom}_A(M,\mrm{L}\Lambda_{\a} (N)) 
\]
for any $M,N \in \mrm{D}(\opn{Mod} A)$. 

Both the infinite dual Koszul complex and the telescope complex enjoy the following base change property: if $A$ is a ring, $\mathbf{a}$ is a finite sequence of elements in $A$, $A \to B$ is a ring map, and $\mathbf{b}$ is the image of $\mathbf{a}$ under this map, then there are isomorphisms
\[
\opn{K}^{\vee}_{\infty}(A; \mathbf{a}) \otimes_A B \cong \opn{K}^{\vee}_{\infty}(B; \mathbf{b}), \quad \opn{Tel}(A; \mathbf{a}) \otimes_A B \cong \opn{Tel}(B; \mathbf{b})
\]
of complexes of $B$-modules.

For any complex $M \in \mrm{D}(\opn{Mod} A)$, there are canonical maps 
\begin{equation}\label{eqn:tor}
\mrm{R}\Gamma_{\a} (M) \to M,
\end{equation}
and 
\begin{equation}\label{eqn:com}
M \to \mrm{L}\Lambda_{\a} (M).
\end{equation}
The complex $M$ is called  cohomologically $\a$-torsion (respectively cohomologically $\a$-adically complete) if the map (\ref{eqn:tor}) (resp. the map (\ref{eqn:com})) is an isomorphism. If $\a$ is weakly proregular then the functors $\mrm{R}\Gamma_{\a}$ and $\mrm{L}\Lambda_{\a}$ are idempotent (\cite[Corollary 4.30]{PSY1}, \cite[Proposition 7.10]{PSY1}), and it follows that in this case the collection of all cohomologically $\a$-torsion (resp. cohomologically $\a$-adically complete) complexes is a triangulated subcategory of $\mrm{D}(\opn{Mod} A)$ which is equal to the essential image of the functor $\mrm{R}\Gamma_{\a}$ (resp. $\mrm{L}\Lambda_{\a}$). Moreover, by \cite[Theorem 7.11]{PSY1}, in this case these categories are equivalent (the Matlis-Greenlees-May equivalence).

\subsection{Hochschild cohomology}

Let $\k$ be a commutative ring, and let $A$ be a commutative $\k$-algebra. We let
\[
A^{\otimes^n_{\k}} := \underbrace{A\otimes_{\k} \dots \otimes_{\k} A}_{n},
\]
and denote by $\mathcal{B}$ the bar resolution 
\[
\dots \to A^{\otimes^n_{\k}} \to \dots \to A^{\otimes^2_{\k}} \to A \to 0.
\]
Given an $A$-bimodule $M$, the $n$-th Hochschild cohomology module of $A$ over $\k$ with coefficients in $M$ is given by 
\[
\mrm{HH}^n(A/\k,M) := H^n\opn{Hom}_{A\otimes_{\k} A}(\mathcal{B},M).
\]
See \cite[Chapter IX]{CE}, \cite[Chapter 1]{Lo} and \cite[Chapter 9]{We} for more details on this classical construction. If $A$ is projective (respectively flat) over $\k$, then $\mathcal{B}$ is a projective (resp. flat) resolution of $A$ over the enveloping algebra $A\otimes_{\k} A$. Hence, in the projective case, the natural map
\[
\mrm{HH}^n(A/\k,M) \to \opn{Ext}^n_{A\otimes_{\k} A}(A,M)
\]
is an isomorphism. When $A$ is only flat over $\k$, but not necessarily projective, this map might fail in general to be an isomorphism. Nevertheless, the modules on the right hand side are interesting on their own, and are sometimes referred to in the literature as the derived Hochschild (or Shukla) cohomology modules of $A$ over $\k$. In this paper we will focus mostly on these modules\footnote{See however Corollary \ref{cor:Answer-To_BF}  and Theorem \ref{thm:main} where even in the possible absence of projectivity we will discuss classical Hochschild cohomology.}, and a bit more generally, on the complex $\mrm{R}\opn{Hom}_{A\otimes_{\k} A}(A,M)$. 
Somewhat imprecisely, we will refer to 
\[
\mrm{R}\opn{Hom}_{A\otimes_{\k} A}(A,M)
\]
as the Hochschild complex of $A$ with coefficients in $M$ even when $A$ is only flat over $\k$. We will however use the notation $\mrm{HH}^n(A/\k,M)$ to denote only the classical Hochschild cohomology modules.

\section{Weak proregularity and flat base change}\label{section:WPR}

Let $\k$ be a base commutative ring, and let $A,B$ be two flat $\k$-algebras. Assume that $A$ and $B$ are equipped with adic topologies, generated by finitely generated ideals $\a\subseteq A$ and $\b \subseteq B$. In that case, the tensor product $A\otimes_{\k} B$ is also naturally equipped with an adic topology. It is generated by the finitely generated ideal $\a\otimes_{\k} B + A\otimes_{\k} \b \subseteq A\otimes_{\k} B$. The aim of this section is to discuss the question: when is this ideal weakly proregular? We allow $A$ to be different from $B$, although we will only use the case $A=B$ in the rest of the paper.

Recall that a ring $\k$ is called absolutely flat (or Von Neumann regular) if every $\k$-module is flat. Over such rings, the above question is easy:

\begin{prop}\label{prop:wpr-over-abflat}
Let $\k$ be an absolutely flat ring. Let $A, B$ be two $\k$-algebras, and let $\a\subseteq A$ and $\b\subseteq B$ be weakly proregular ideals. Then the ideal 
\[
\a\otimes_{\k} B + A\otimes_{\k} \b \subseteq A\otimes_{\k} B
\]
is weakly proregular.
\end{prop}
\begin{proof}
In case where $\k$ is a field, and $A$ and $B$ are noetherian and complete with respect to the adic topology, this is shown in \cite[Example 4.35]{PSY1}, and the proof there remains true under the above assumptions.
\end{proof}

\begin{rem}\label{remark-wpr-flat}
Assume $A$ is a ring, $\a\subseteq A$ is a weakly proregular ideal, $B$ is a flat $A$-algebra, and $\b = \a\cdot B$. Then by \cite[Example 3.0(B)]{AJL1}, the ideal $\b$ is also weakly proregular.
\end{rem}

Recall that if $A$ is a noetherian ring, $\a\subseteq A$ an ideal, and if $I$ is an injective $A$-module, then $\Gamma_{\a}(I)$ is also an injective $A$-module. (For example by \cite[Lemma 3.2]{Ha}). We now state and prove a weaker form of this fact in the case when $A$ is not necessarily noetherian, but $\Lambda_{\a}(A)$ is.

If $A$ is a ring, $\a\subseteq A$ a finitely generated ideal, and $M$ is an $A$-module, then $M$ is called $\a$-flasque if for each $k>0$, we have that $\mrm{H}_{\a}^k(M) = 0 $, where $\mrm{H}_{\a}^k(M):= \mrm{H}^k(\mrm{R}\Gamma_{\a}(M))$. Any injective module is $\a$-flasque. If $M$ is $\a$-flasque, then the canonical morphism $\Gamma_{\a} (M) \to \mrm{R}\Gamma_{\a} (M)$ is an isomorphism. By \cite[Theorem 3.4.10]{BS}, the direct limit of $\a$-flasque modules is $\a$-flasque.

\begin{lem}\label{lem-tor-of-inj-is-flasque}
Let $A$ be a ring, and let $\a, \b\subseteq A$ be two finitely generated ideals. Suppose that the ring $\widehat{A} = \Lambda_{\mfrak{a}}(A)$ is noetherian. Let $\widehat{\mfrak{b}} = \mfrak{b}\widehat{A}$. Then for any injective $A$-module $I$, the $\widehat{A}$-module $\widehat{\Gamma}_{\mfrak{a}} I$ is $\widehat{\mfrak{b}}$-flasque.
\end{lem}
\begin{proof}
Let $A_j = A/\mfrak{a}^{j+1}$. Since $\mfrak{a}$ is finitely generated, there is an isomorphism $A_j \cong \widehat{A}/(\mfrak{a}\widehat{A})^{j+1}$. Note that by assumption, $A_j$ is noetherian. Let $\widehat{\mfrak{b}}_j$ be the image of $\widehat{\mfrak{b}}$ in $A_j$. Let $I_j = \opn{Hom}_A(A_j,I)$. Then $\widehat{\Gamma}_{\mfrak{a}} I = \varinjlim I_j$, so it is enough to show that $I_j$ is $\widehat{\mfrak{b}}$-flasque. Note also that $I_j$ is an injective $A_j$-module. Let $k>0$,  
let $\widehat{\mathbf{b}}$ be a finite sequence generating $\widehat{\mfrak{b}}$, and let $\mathbf{b}_j$ be its image in $A_j$.
Since $\widehat{A}$ is noetherian, $\widehat{\mfrak{b}}$ is weakly proregular, so that
\[
H^k_{\widehat{\mfrak{b}}} (I_j) \cong H^k(\opn{K}^{\vee}_{\infty}(\widehat{A};\widehat{\mathbf{b}}) \otimes_{\widehat{A}} I_j) \cong
H^k(\opn{K}^{\vee}_{\infty}(A_j;\mathbf{b}_j)\otimes_{A_j} I_j) \cong
H^k_{\widehat{\mfrak{b}}_j}(I_j)
\]
where the last isomorphism follows from the fact that $A_j$ is noetherian, so that $\widehat{\mfrak{b}}_j$ is weakly proregular. Since $I_j$ is injective over $A_j$, it follows that $H^k_{\widehat{\mfrak{b}}_j}(I_j) = 0$ for all $k>0$, which proves the claim. 
\end{proof}

\begin{prop}\label{wpr-prop}
Let $A$ be a commutative ring, and let $\a\subseteq A$ be a weakly proregular ideal such that $\Lambda_{\a}(A)$ is noetherian. Let $\b\subseteq A$ be an ideal containing $\a$. Then $\b$ is also weakly proregular.
\end{prop}
\begin{proof}
We keep the notations of Lemma \ref{lem-tor-of-inj-is-flasque}. It is clear that $A/\b$ is noetherian.
Let $\mathbf{a}$ be a finite sequence generating $\mfrak{a}$, and let $\mathbf{b}$ be a finite sequence generating $\mfrak{b}$. Let $I$ be an injective $A$-module. By \cite[Theorem 1.1]{Sc}, It is enough to show that
\[
H^k(\opn{Tel}(A;\mathbf{b}) \otimes_A I) = 0
\]
for all $k \ne 0$. 
\par
Since $\mfrak{a}\subseteq\mfrak{b}$, the ideal generated by the concatenated sequence $(\mathbf{a},\mathbf{b})$ is equal to the ideal generated by $\mathbf{b}$, so there is a homotopy equivalence
$\opn{Tel}(A;(\mathbf{a},\mathbf{b})) \cong \opn{Tel}(A;\mathbf{b})$.
Hence, there is an isomorphism in $\mrm{D}(\opn{Mod} A)$
\[
\opn{Tel}(A;\mathbf{b}) \otimes_A I \cong
\opn{Tel}(A;(\mathbf{a},\mathbf{b})) \otimes_A I \cong
\opn{Tel}(A;\mathbf{b}) \otimes_A \opn{Tel}(A;\mathbf{a}) \otimes_A I 
\]
Since $\mathbf{a}$ is a weakly proregular sequence, $I$ is an injective $A$-module, and $\opn{Tel}(A;\mathbf{b})$ is a bounded complex of flat modules, the latter is isomorphic in $\mrm{D}(\opn{Mod} A)$ to
\[
\opn{Tel}(A;\mathbf{b}) \otimes_A \Gamma_{\mfrak{a}} I
\]
Thus, it is enough to show that all the cohomologies (except the zeroth) of the complex of $A$-modules $\opn{Tel}(A;\mathbf{b}) \otimes_A \Gamma_{\mfrak{a}} I$ vanish. Let 
\[
\opn{Rest}_{\widehat{A}/A}:\mrm{D}(\opn{Mod} \widehat{A}) \to \mrm{D}(\opn{Mod} A)
\]
be the forgetful functor, and let  $\widehat{\mathbf{b}}$ be the image of the sequence $\mathbf{b}$ in $\widehat{A}$. Consider the complex 
\[
\opn{Tel}(\widehat{A};\widehat{\mathbf{b}}) \otimes_{\widehat{A}} \widehat{\Gamma}_{\a} I \in \mrm{D}(\opn{Mod} \widehat{A}).
\]
We claim that 
\begin{equation}\label{eqn:Q-calc}
\opn{Rest}_{\widehat{A}/A}(\opn{Tel}(\widehat{A};\widehat{\mathbf{b}}) \otimes_{\widehat{A}} \widehat{\Gamma}_{\a} I ) =  \opn{Tel}(A;\mathbf{b}) \otimes_A \Gamma_{\mfrak{a}} I.
\end{equation}
Indeed, by the base change property of the telescope complex, we have an isomorphism of complexes in $\mrm{D}(\opn{Mod} \widehat{A})$:
\[
\opn{Tel}(\widehat{A};\widehat{\mathbf{b}}) \otimes_{\widehat{A}} \widehat{\Gamma}_{\a} I \cong \opn{Tel}(A;\mathbf{b}) \otimes_A \widehat{\Gamma}_{\a} I.
\]
So using the fact that 
\[
\opn{Rest}_{\widehat{A}/A}(\widehat{\Gamma}_{\a}I) = \Gamma_{\a} I
\]
we obtain equation (\ref{eqn:Q-calc}). Since for a complex $M \in \mrm{D}(\opn{Mod} \widehat{A})$, we have that 
\[
H^k(M) = 0 \quad\mbox{if and only if}\quad H^k(\opn{Rest}_{\widehat{A}/A}(M)) = 0,
\]
it is thus enough to show that $H^k(\opn{Tel}(\widehat{A};\widehat{\mathbf{b}}) \otimes_{\widehat{A}} \widehat{\Gamma}_{\a} I) = 0$ for all $k \ne 0$. By weak proregularity of the sequence $\widehat{\mathbf{b}}$, there is an isomorphism in $\mrm{D}(\opn{Mod}\widehat{A})$:
\[
\opn{Tel}(\widehat{A};\widehat{\mathbf{b}}) \otimes_{\widehat{A}} \widehat{\Gamma}_{\mfrak{a}}I \cong
\mrm{R}\Gamma_{\widehat{\mfrak{b}}} \widehat{\Gamma}_{\mfrak{a}}I
\]
By Lemma \ref{lem-tor-of-inj-is-flasque}, this is isomorphic in $\mrm{D}(\opn{Mod}\widehat{A})$ to
\[
\Gamma_{\widehat{\mfrak{b}}} \widehat{\Gamma}_{\mfrak{a}}I
\]
Since this complex is clearly concentrated in degree zero, it follows that all of its cohomologies except the zeroth vanish, which proves the result.
\end{proof}

Here is the main result of this section:
\begin{thm}\label{thm:formally-finite-type-wpr}
Let $\k$ be a commutative ring, let $A$ be a flat noetherian $\k$-algebra, and let $\a\subseteq A$ be an ideal such that $A/\a$ is essentially of finite type over $\k$. Let $B$ be a flat noetherian $\k$-algebra, and let $\b\subseteq B$ be an ideal. Then the ideal 
\[
I:= \a\otimes_{\k} B + A\otimes_{\k} \b\subseteq A\otimes_{\k} B
\] 
is weakly proregular.
\end{thm}
\begin{proof}
According to Remark \ref{remark-wpr-flat}, the ideal $I_1 := \a\otimes_{\k} B \subseteq A\otimes_{\k} B$ is weakly proregular. Since $B$ is flat over $\k$, we have that $(A\otimes_{\k} B)/I_1 \cong A/\a \otimes_{\k} B$, and as
$A/\a$ is essentially of finite type over $\k$, it follows that $(A\otimes_{\k} B)/I_1$ is noetherian. Hence,  $\Lambda_{I_1} (A\otimes_{\k} B)$ is also noetherian. Since $I_1 \subseteq I$, the result follows from Proposition \ref{wpr-prop}.
\end{proof}

\begin{rem}
The assumption that $A$ is noetherian in the above result can be relaxed: it is enough to assume that $\a$ is weakly proregular. It is an open problem to the author if the above result remains true without the assumption that $A/\a$ is essentially of finite type over $\k$ (as in Proposition \ref{prop:wpr-over-abflat}).
\end{rem}

\begin{rem}
Let $\k$ be a commutative ring, and let $A,B$ be two commutative noetherian $\k$-algebras which are adically complete with respect to ideals $\a\subseteq A, \b\subseteq B$. In this situation, Grothendieck \cite[Section 10.7]{EGA} defined the fiber product of the two affine formal schemes $\opn{Spf} A, \opn{Spf} B$ over $\k$ to be the formal spectrum of the ring $\Lambda_{\a\otimes_{\k} B + A\otimes_{\k} \b}(A\otimes_{\k} B)$.

Now, we switch to the point of view of derived algebraic geometry, and assume that $A$ and $B$ are flat over $\k$. Forgetting the adic structure on $A,B$, the flatness assumption ensures that in this situation, the usual fiber product $\opn{Spec}(A\otimes_{\k} B)$ of the schemes $\opn{Spec}(A)$ and $\opn{Spec}(B)$ coincides with their derived fiber product. Returning to the adic situation, Lurie defined in \cite[Section 4.2]{Lu} a notion of a derived completion of an ($E_{\infty}$) ring, and showed (\cite[Section 4.3]{Lu}) that if the ring is noetherian then its derived completion coincides with its ordinary completion. Recently, using our \cite[Theorem 4.2]{PSY3}, it was shown in \cite[Proposition 5.4]{BGW}, that if $R$ is any commutative ring, and if $I\subseteq R$ is a weakly proregular ideal, then the derived $I$-completion of $R$ coincides with its ordinary $I$-adic completion. Hence, the results of this section implies the following:  
\end{rem}

\begin{cor}\label{cor:derived-fiber-of-formal}
Let $\k$ be a commutative ring, and let $A,B$ be noetherian flat $\k$-algebras which are adically complete with respect to ideals $\a\subseteq A, \b\subseteq B$. Assume further that either $\k$ is an absolutely flat ring (e.g, a field) or that $A/\a$ is essentially of finite type over $\k$.
Then the derived fiber product of the formal schemes $\opn{Spf} A, \opn{Spf} B$ over $\k$ is equal to the formal spectrum of $\Lambda_{\a\otimes_{\k} B + A\otimes_{\k} \b}(A\otimes_{\k} B)$.
\end{cor}

\section{The functors $\mrm{R}\widehat{\Gamma}_{\a},\mrm{L}\widehat{\Lambda}_{\a}$}\label{section:widehat}

Let $A$ be a commutative ring, let $\a\subseteq A$ be a finitely generated ideal, and let $\widehat{A}$ be the $\a$-adic completion of $A$. For any $A$-module $M$, the $A$-modules $\Gamma_{\a}(M)$ and $\Lambda_{\a}(M)$ carry naturally $\widehat{A}$-module structures, and one obtain additive functors $\widehat{\Gamma}_{\a}, \widehat{\Lambda}_{\a}: \opn{Mod} A \to \opn{Mod} \widehat{A}$, defined by the same formulas as $\Gamma_{\a}$ and $\Lambda_{\a}$. These functors have derived functors $\mrm{R}\widehat{\Gamma}_{\a}, \mrm{L}\widehat{\Lambda}_{\a}: \mrm{D}(\opn{Mod} A) \to \mrm{D}(\opn{Mod} \widehat{A})$ calculated using K-injective and K-flat resolutions respectively. This section is dedicated to a study of these functors. 

Keeping an eye towards the main goal of this text, we must avoid assuming that $A$ is noetherian. Hence, we do not know if the completion map $A\to \widehat{A}$ is flat. We overcome this issue by using DG-algebras. DG-algebras will be assumed to be (graded-)commutative. We refer the reader to \cite{Av,Ke,ML,Ye2} for information about DG-algebras and their derived categories. For a DG-algebra $A$, we denote by $\opn{DGMod} A$ the  category of DG-modules over $A$, and by $\mrm{D}(\opn{DGMod} A)$ the derived category over $A$.

 We shall need the following well known result from DG-homological algebra:

\begin{prop}\label{prop:qi-keep-info}
Let $A\to B$ be a quasi-isomorphism between two commutative DG-algebras, and let 
\[
\opn{Rest}_{B/A}:\mrm{D}(\opn{DGMod} B) \to \mrm{D}(\opn{DGMod} A)
\]
be the forgetful functor.
\begin{enumerate}
\item There is an isomorphism
\[
1_{\mrm{D}(\opn{DGMod} B)} \cong B\otimes^{\mrm{L}}_A \opn{Rest}_{B/A}(-)
\]
of functors $\mrm{D}(\opn{DGMod} B) \to \mrm{D}(\opn{DGMod} B)$.
\item There is an isomorphism
\[
1_{\mrm{D}(\opn{DGMod} B)} \cong \mrm{R}\opn{Hom}_A(B,\opn{Rest}_{B/A}(-))
\]
of functors $\mrm{D}(\opn{DGMod} B) \to \mrm{D}(\opn{DGMod} B)$.
\end{enumerate}
\end{prop}
\begin{proof}
\begin{enumerate}
\item This follows immediately from \cite[Tag 09S6]{Stack}, or \cite[Proposition 2.5(1)]{Ye2}.
\item This follows immediately from \cite[Lemma 2.2]{Sh}, or \cite[Proposition 2.5(2)]{Ye2}.
\end{enumerate}
\end{proof}

As far as we know, the next results are new even in the case where $A$ is noetherian. In the noetherian case, one does not need DG-algebras in the proof of the next result.

\begin{thm}\label{thm:wide-rgamma}
Let $A$ be a commutative ring, let $\a\subseteq A$ be a finitely generated ideal, and let $\mathbf{a}$ be a finite sequence that generates $\a$. Assume that $\a$ is weakly proregular.
Then there is an isomorphism of functors
\[
\mrm{R}\widehat{\Gamma}_{\a} (-) \cong \widehat{A} \otimes^{\mrm{L}}_A  (\opn{K}^{\vee}_{\infty}(A;\mathbf{a}) \otimes_{A} -).
\]
\end{thm}
\begin{proof}
Set $\widehat{A} := \Lambda_{\a} (A)$. Consider the completion map $A\to \widehat{A}$. As $A$ is not necessarily noetherian, this map might fail to be flat, so let $A \xrightarrow{f} \widetilde{A} \xrightarrow{g} \widehat{A}$ be a K-flat DG-algebra resolution of $A\to \widehat{A}$. That is, $f:A \to \widetilde{A}$ is a K-flat DG-algebra map, $g:\widetilde{A} \to \widehat{A}$ is a quasi-isomorphism of DG-algebras, and $g\circ f$ is equal to the completion map $A \to \widehat{A}$. We denote by 
\[
\opn{Rest}_{\widehat{A}/\widetilde{A}} :\mrm{D}(\opn{Mod} \widehat{A}) \to \mrm{D}(\opn{DGMod} \widetilde{A})
\]
and by
\[
\opn{Rest}_{\widetilde{A}/A} : \mrm{D}(\opn{DGMod} \widetilde{A}) \to \mrm{D}(\opn{Mod} A)
\]
the corresponding forgetful functors.
Set 
\[
\mrm{R}\widetilde{\Gamma}_{\a} (-) := \opn{Rest}_{\widehat{A}/\widetilde{A}} \circ \mrm{R}\widehat{\Gamma}_{\a}(-) : \mrm{D}(\opn{Mod} A) \to \mrm{D}(\opn{DGMod} \widetilde{A}).
\]
Let $M \in \mrm{D}(\opn{Mod} A)$. Let $P\to M$ be a K-flat resolution of $M$, and let $M\to I$ be a K-injective resolution of $M$. The map $f:A\to \widetilde{A}$ induces a map $1_P \otimes_A f: P \to P\otimes_A \widetilde{A}$. Let $P\otimes_A \widetilde{A} \to J$ be a K-injective resolution of $P\otimes_A \widetilde{A}$ over $\widetilde{A}$. Because $f$ is flat, $J$ is also a K-injective resolution of $P\otimes_A \widetilde{A}$ over $A$. There is a unique map $\phi:I \to J$ in $\mrm{K}(\opn{Mod} A)$, which makes the diagram
\begin{equation}\label{eqn:inj-square}
\xymatrixcolsep{5pc}
\xymatrix{
P \ar[r]^{1_P\otimes_A f} \ar[d] & P\otimes_A \widetilde{A} \ar[d] \\
I \ar[r]^{\phi} & J
}
\end{equation}
commutative, 
and it induces a map $\Gamma_{\a}(\phi): \Gamma_{\a} (I) \to \Gamma_{\a} (J)$. Our goal is to show that $\Gamma_{\a}(\phi)$ is a quasi-isomorphism. The morphism of functors 
\[
\alpha(-): \Gamma_{\a} (-) \to \opn{K}^{\vee}_{\infty}(A;\mathbf{a}) \otimes_A -
\]
that was constructed in \cite[Equation (4.19)]{PSY1}, and the map $\phi$ induce the commutative diagram:
\begin{equation}\label{eqn:main-inj-square}
\xymatrixcolsep{5pc} 
\xymatrix{
\Gamma_{\a} (I) \ar[r]^{\Gamma_{\a}(\phi)} \ar[d]^{\alpha_I} & \Gamma_{\a} (J) \ar[d]^{\alpha_J} \\
\opn{K}^{\vee}_{\infty}(A;\mathbf{a}) \otimes_A I \ar[r]^{1\otimes_A \phi} & \opn{K}^{\vee}_{\infty}(A;\mathbf{a}) \otimes_A J
}
\end{equation}
Because $\mathbf{a}$ is weakly proregular, the two vertical maps are quasi-isomorphisms. We claim that the bottom horizontal map is also a quasi-isomorphism. To see this, consider the commutative diagram in $\mrm{K}(\opn{Mod} A)$:
\[
\xymatrixcolsep{5pc}
\xymatrix{
\opn{K}^{\vee}_{\infty}(A;\mathbf{a}) \otimes_A I \ar[r]^{1} & \opn{K}^{\vee}_{\infty}(A;\mathbf{a}) \otimes_A J \\
\opn{K}^{\vee}_{\infty}(A;\mathbf{a}) \otimes_A P \ar[u]^{2} \ar[r]^{3} & \opn{K}^{\vee}_{\infty}(A;\mathbf{a}) \otimes_A P\otimes_A \widetilde{A}  \ar[u]^{4} \\
\opn{Tel}(A;\mathbf{a}) \otimes_A P \ar[r]^{5} \ar[u]^{6}\ar[d]^{7}& \opn{Tel}(A;\mathbf{a}) \otimes_A P \otimes_A \widetilde{A} \ar[u]^{8}\ar[d]^{9}\\
\opn{Tel}(A;\mathbf{a}) \otimes_A P \otimes_A \opn{Hom}_A(\opn{Tel}(A;\mathbf{a}),A) \ar[r]^{10} & \opn{Tel}(A;\mathbf{a}) \otimes_A P \otimes_A \widehat{A}
}
\]
The top square in this diagram is induced from the square (\ref{eqn:inj-square}), the middle square is induced from the quasi-isomorphism 
\[
\opn{Tel}(A;\mathbf{a}) \to \opn{K}^{\vee}_{\infty}(A;\mathbf{a}),
\]
and the bottom square is induced from the commutative diagram \cite[Equation 5.26]{PSY1}. By \cite[Corollary 5.23]{PSY1}, the map $\opn{Hom}_A(\opn{Tel}(A;\mathbf{a}),A) \to \widehat{A}$ is a quasi-isomorphism. Since $\opn{Tel}(A;\mathbf{a})$ and $P$ are both K-flat, it follows that (10) is a quasi-isomorphism. K-flatness of $P$ also implies that (9) is a quasi-isomorphism. The map (7), which is induced by the map $A \to \opn{Hom}_A(\opn{Tel}(A;\mathbf{a}),A)$ is a quasi-isomorphism by \cite[Corollary after (0.3)*]{AJL1} (or the proof of \cite[Lemma 7.6]{PSY1}). Hence, the map (5) is also a quasi-isomorphism. It is clear that the maps (6) and (8) are quasi-isomorphisms, so that (3) is also a quasi-isomorphism. As (2) and (4) are also quasi-isomorphisms, we deduce that (1) is a quasi-isomorphism. Returning to the commutative diagram (\ref{eqn:main-inj-square}), we deduce that the map 
\[
\Gamma_{\a}(\phi): \Gamma_{\a}(I) \to \Gamma_{\a}(J)
\]
is a quasi-isomorphism. 

There are functorial isomorphisms in $\mrm{D}(\opn{DGMod} \widetilde{A})$:
\[
\mrm{R}\widetilde{\Gamma}_{\a} (M) = \opn{Rest}_{\widehat{A}/\widetilde{A}} (\mrm{R}\widehat{\Gamma}_{\a} (M)) \cong \opn{Rest}_{\widehat{A}/\widetilde{A}}(\widehat{\Gamma}_{\a} (I)).
\]
Since the map $\Gamma_{\a}(\phi): \Gamma_{\a}(I) \to \Gamma_{\a}(J)$ is a quasi-isomorphism, it follows that the map
\[
\opn{Rest}_{\widehat{A}/\widetilde{A}} (\widehat{\Gamma}_{\a}(\phi)) :
\opn{Rest}_{\widehat{A}/\widetilde{A}} (\widehat{\Gamma}_{\a} (I)) \to \opn{Rest}_{\widehat{A}/\widetilde{A}}(\widehat{\Gamma}_{\a} (J))
\]
is also a quasi-isomorphism. The DG $\widetilde{A}$-module $\opn{Rest}_{\widehat{A}/\widetilde{A}}(\widehat{\Gamma}_{\a} (J))$ is a sub DG-module of $J$, and the inclusion map induces a map
\begin{equation}\label{eqn:incl}
\opn{Rest}_{\widehat{A}/\widetilde{A}}(\widehat{\Gamma}_{\a} (J)) \to \opn{K}^{\vee}_{\infty}(A;\mathbf{a}) \otimes_{A} J.
\end{equation}
Applying the forgetful functor $\opn{Rest}_{\widetilde{A}/A}$ to the map in Equation (\ref{eqn:incl}) yields the quasi-isomorphism
\[
\Gamma_{\a} (J) \cong \opn{K}^{\vee}_{\infty}(A;\mathbf{a}) \otimes_{A} J,
\]
so that the map in Equation (\ref{eqn:incl}) is also a quasi-isomorphism.
Hence,
\[
\mrm{R}\widetilde{\Gamma}_{\a} (M) \cong \opn{K}^{\vee}_{\infty}(A;\mathbf{a}) \otimes_A J \cong \opn{K}^{\vee}_{\infty}(A;\mathbf{a}) \otimes_{A} M \otimes_A \widetilde{A}.
\]
By Proposition \ref{prop:qi-keep-info}, there is an isomorphism of functors 
\[
1_{D(\opn{DGMod} \widehat{A})} \cong \widehat{A} \otimes^{\mrm{L}}_{\widetilde{A}} \opn{Rest}_{\widehat{A}/\widetilde{A}}(-).
\] 
Hence, 
\[
\mrm{R}\widehat{\Gamma}_{\a} (-) \cong \widehat{A} \otimes^{\mrm{L}}_{\widetilde{A}} \mrm{R}\widetilde{\Gamma}_{\a} (-),
\]
which implies that
\[
\mrm{R}\widehat{\Gamma}_{\a} (M) \cong  \widehat{A} \otimes^{\mrm{L}}_A (\opn{K}^{\vee}_{\infty}(A;\mathbf{a}) \otimes_{A} M).
\]
\end{proof}

Dually, we have the next result for the $\mrm{L}\widehat{\Lambda}_{\a}$ functor. Note however that in this case, even if $A$ is noetherian, we have to use DG-algebra resolutions, because $\widehat{A}$ is almost never projective over $A$ (See for example \cite[Theorem 2.1]{BF}).

\begin{thm}\label{thm:wide-llambda}
Let $A$ be a commutative ring, let $\a\subseteq A$ be a finitely generated ideal, and let $\mathbf{a}$ be a finite sequence that generates $\a$. Assume that $\a$ is weakly proregular.
Then there is an isomorphism of functors
\[
\mrm{L}\widehat{\Lambda}_{\a} (-) \cong \mrm{R}\opn{Hom}_A(\widehat{A}\otimes_A \opn{Tel}(A;\mathbf{a}),-)
\]
\end{thm}
\begin{proof}
We use notations as in the proof of Theorem \ref{thm:wide-rgamma}. Let $A \xrightarrow{f} \widetilde{A} \xrightarrow{g} \widehat{A}$ be a K-projective DG-algebra resolution of $A\to \widehat{A}$, and set 
\[
\mrm{L}\widetilde{\Lambda}_{\a} (-) := \opn{Rest}_{\widehat{A}/\widetilde{A}} \circ \mrm{L}\widehat{\Lambda}_{\a}(-) : \mrm{D}(\opn{Mod} A) \to \mrm{D}(\opn{DGMod} \widetilde{A}).
\]
Let $M \in \mrm{D}(\opn{Mod} A)$. Let $P \to M$ be a K-projective resolution, and let $M \to I$ be a K-injective resolution. The map $f: A\to \widetilde{A}$ induces a map $\opn{Hom}_A(f,1): \opn{Hom}_A(\widetilde{A},I) \to I$. Let $Q \to \opn{Hom}_A(\widetilde{A},I)$ be a K-projective resolution of $\opn{Hom}_A(\widetilde{A},I)$ over $\widetilde{A}$. There is a unique map $\phi: Q \to P$ in $\mrm{K}(\opn{Mod} A)$ making the diagram
\begin{equation}\label{eqn:proj-square}
\xymatrixcolsep{6pc}
\xymatrix{
\opn{Hom}_A(\widetilde{A},I) \ar[r]^{\opn{Hom}_A(f,1)} & I \\
Q \ar[u]\ar[r]^{\phi} & P \ar[u]
}
\end{equation}
commutative. The morphism of functors 
\[
\beta(-): \opn{Hom}_A(\opn{Tel}(A;\mathbf{a}),-) \to \Lambda_{\a} (-) 
\]
that was constructed in \cite[Definition 5.16]{PSY1} and the map $\phi:Q\to P$ induce a commutative diagram
\begin{equation}\label{eqn:main-proj-square}
\xymatrixcolsep{6pc}
\xymatrix{
\Lambda_{\a}(Q) \ar[r]^{\Lambda_{\a}(\phi)} & \Lambda_{\a}(P)\\
\opn{Hom}_A(\opn{Tel}(A;\mathbf{a}),Q) \ar[u]^{\beta_Q}\ar[r]^{\opn{Hom}_A(1,\phi)} & \opn{Hom}_A(\opn{Tel}(A;\mathbf{a}),P)\ar[u]^{\beta_P},
}
\end{equation}
and because of weak proregularity of $\mathbf{a}$, the two vertical maps in this diagram are quasi-isomorphisms. We will show that the bottom horizontal map is also a quasi-isomorphism, which will imply that the top horizontal map is a quasi-isomorphism. To see this, consider the commutative diagram in $\mrm{K}(\opn{Mod} A)$:
\[
\xymatrixcolsep{5pc}
\xymatrix{
\opn{Hom}_A(T,Q) \ar[r]^{1}\ar[d]^{2} & \opn{Hom}_A(T,P)\ar[d]^{3}\\
\opn{Hom}_A(T,\opn{Hom}_A(\widetilde{A},I)) \ar[r]^{4} & \opn{Hom}_A(\opn{Tel}(A;\mathbf{a}),I)\\
\opn{Hom}_A(T,\opn{Hom}_A(\widehat{A},I)) \ar[u]^{5} \ar[r]^{6} & \opn{Hom}_A(T, \opn{Hom}_A( \opn{Hom}_A(T,A) ,I)) \ar[u]^{7}
}
\]
where we have set $T := \opn{Tel}(A;\mathbf{a})$. The top square of this diagram is induced from the square (\ref{eqn:proj-square}), while the bottom square is induced from the commutative diagram of \cite[Equation 5.26]{PSY1}. Weak proregularity of $\mathbf{a}$ implies that the map (6) is a quasi-isomorphism. The fact that $I$ is K-injective and that $\opn{Tel}(A;\mathbf{a})$ is K-projective implies that (5) is a quasi-isomorphism. The hom-tensor adjunction and \cite[Corollary after (0.3)*]{AJL1} (or the proof of \cite[Lemma 7.6]{PSY1}) shows that (7) is also a quasi-isomorphism. Hence, (4) is a quasi-isomorphism. As (2) and (3) are clearly quasi-isomorphisms, we deduce that (1) is a quasi-isomorphism. Returning to the commutative diagram (\ref{eqn:main-proj-square}), we deduce that the map
\[
\Lambda_{\a}(\phi): \Lambda_{\a}(Q) \to \Lambda_{\a}(P)
\]
is a quasi-isomorphism. 

There are functorial isomorphisms in $\mrm{D}(\opn{DGMod} \widetilde{A})$:
\[
\mrm{L}\widetilde{\Lambda}_{\a} (M) = \opn{Rest}_{\widehat{A}/\widetilde{A}} (\mrm{L}\widehat{\Lambda}_{\a} (M)) \cong 
\opn{Rest}_{\widehat{A}/\widetilde{A}} (\widehat{\Lambda}_{\a}(P)),
\]
and since the map $\Lambda_{\a}(\phi) : \Lambda_{\a}(Q) \to \Lambda_{\a}(P)$ is a quasi-isomorphism, it follows that the map
\[
\opn{Rest}_{\widehat{A}/\widetilde{A}}( \widehat{\Lambda}_{\a}(\phi) ): \opn{Rest}_{\widehat{A}/\widetilde{A}} (\widehat{\Lambda}_{\a}(Q))  \to \opn{Rest}_{\widehat{A}/\widetilde{A}}(\widehat{\Lambda}_{\a}(P))
\]
is also a quasi-isomorphism. By \cite[Corollary 5.23]{PSY1}, there is an $A$-linear quasi-isomorphism
\[
\beta_Q: \opn{Hom}_A(\opn{Tel}(A;\mathbf{a}),Q) \to \Lambda_{\a}(Q),
\]
and it is easy to verify that the same construction give rise to an $\widetilde{A}$-linear quasi-isomorphism
\[
\opn{Hom}_A(\opn{Tel}(A;\mathbf{a}),Q) \to \opn{Rest}_{\widehat{A}/\widetilde{A}} (\widehat{\Lambda}_{\a}(Q)).
\]

Hence, there are isomorphisms of functors
\begin{gather*}
\mrm{L}\widetilde{\Lambda}_{\a} (M) \cong \opn{Hom}_A(\opn{Tel}(A;\mathbf{a}), Q) \cong\nonumber\\ \opn{Hom}_A(\opn{Tel}(A;\mathbf{a}),\opn{Hom}_A(\widetilde{A},I)) \cong \mrm{R}\opn{Hom}_A(\opn{Tel}(A;\mathbf{a})\otimes_A \widetilde{A}, M)\nonumber.
\end{gather*}
By Proposition \ref{prop:qi-keep-info}, there is an isomorphism of functors 
\[
1_{D(\opn{DGMod} \widehat{A})} \cong \mrm{R}\opn{Hom}_{\widetilde{A}}(\widehat{A} ,\opn{Rest}_{\widehat{A}/\widetilde{A}}(-)).
\]
Hence, 
\[
\mrm{L}\widehat{\Lambda}_{\a} (-) \cong \mrm{R}\opn{Hom}_{\widetilde{A}}(\widehat{A}, \mrm{L}\widetilde{\Lambda}_{\a} (-)),
\]
which implies that
\[
\mrm{L}\widehat{\Lambda}_{\a} (M) \cong \mrm{R}\opn{Hom}_A(\opn{Tel}(A;\mathbf{a}) \otimes_A \widehat{A}, M).
\]
\end{proof}

\begin{cor}\label{cor:generalized-GM}
Let $A$ be a commutative ring, let $\a\subseteq A$ be a weakly proregular ideal, and let $M,N \in \mrm{D}(\opn{Mod} A)$. Then there are isomorphisms
\[
\mrm{L}\widehat{\Lambda}_{\a} (\mrm{R}\opn{Hom}_A(M,N))	 \cong \mrm{R}\opn{Hom}_A( \mrm{R}\widehat{\Gamma}_{\a} (M), N) \cong \mrm{R}\opn{Hom}_A(M,\mrm{L}\widehat{\Lambda}_{\a} (N))
\]
of functors 
\[
\mrm{D}(\opn{Mod} A) \times \mrm{D}(\opn{Mod} A) \to \mrm{D}(\opn{Mod} \widehat{A}).
\]
\end{cor}
\begin{proof}
Let $\mathbf{a}$ be a finite sequence that generates $\a$. The hom-tensor adjunction and the quasi-isomorphism
$\opn{K}^{\vee}_{\infty}(A;\mathbf{a}) \cong \opn{Tel}(A;\mathbf{a})$ show that there are bifunctorial isomorphisms
\[
\mrm{R}\opn{Hom}_A(\opn{Tel}(A;\mathbf{a})\otimes_A \widehat{A}, \mrm{R}\opn{Hom}_A(M,N) ) \cong 
\mrm{R}\opn{Hom}_A(\opn{K}^{\vee}_{\infty}(A;\mathbf{a})\otimes_A \widehat{A} \otimes^{\mrm{L}}_A M, N)
\]
and
\[
\mrm{R}\opn{Hom}_A(\opn{K}^{\vee}_{\infty}(A;\mathbf{a})\otimes_A \widehat{A} \otimes^{\mrm{L}}_A M, N)
\cong
\mrm{R}\opn{Hom}_A(M, \mrm{R}\opn{Hom}_A(\opn{Tel}(A;\mathbf{a})\otimes_A \widehat{A},N)),
\]
so the result follows from Theorem \ref{thm:wide-rgamma}  and Theorem \ref{thm:wide-llambda}.
\end{proof}

\begin{rem}
If $A = \widehat{A}$, then the above corollary collapses to the Greenlees-May duality (See \cite[Theorem (0.3]{AJL1}, or \cite[Theorem 7.12]{PSY1}).
\end{rem}

The next two corollaries will be applied in the next section to study relations between the derived completion functor and Hochschild cohomology.

\begin{cor}\label{cor:compare-adic-ext}
Let $A$ be a commutative ring, and let $\a\subseteq A$ be a weakly proregular ideal. Given a ring map $\widehat{A} \to B$, and a complex $M \in \mrm{D}(\opn{Mod} B)$,  there is an isomorphism
\[
\mrm{R}\opn{Hom}_A(M,\mrm{L}\Lambda_{\a}(-)) \cong \mrm{R}\opn{Hom}_{\widehat{A}}(M,\mrm{L}\widehat{\Lambda}_{\a}(-))
\]
of functors 
\[
\mrm{D}(\opn{Mod} A) \to \mrm{D}(\opn{Mod} B).
\]
\end{cor}
\begin{proof}
Let $\mathbf{a}$ be a finite sequence that generates $\a$. Let $N \in \mrm{D}(\opn{Mod} A)$. By Theorem \ref{thm:wide-llambda}, there is an isomorphism of functors
\[
\mrm{R}\opn{Hom}_{\widehat{A}}(M,\mrm{L}\widehat{\Lambda}_{\a}(N)) \cong \mrm{R}\opn{Hom}_{\widehat{A}}(M, \mrm{R}\opn{Hom}_A(\opn{Tel}(A;\mathbf{a}) \otimes_A \widehat{A}, N)).
\]
Applying the hom-tensor adjunction twice, we get an isomorphism
\[
\mrm{R}\opn{Hom}_{\widehat{A}}(M, \mrm{R}\opn{Hom}_A(\opn{Tel}(A;\mathbf{a}) \otimes_A \widehat{A}, N)) \cong
\mrm{R}\opn{Hom}_A(M, \opn{Hom}_A(\opn{Tel}(A;\mathbf{a}), N)),
\]
which proves the claim.
\end{proof}

\begin{cor}\label{cor:LLam-ofHom}
Let $A$ be a commutative ring, and let $B$ be a commutative $A$-algebra. Let $\a\subseteq A$ be a  weakly proregular ideal, let $\b = \a \cdot B$, and assume that $\b$ is also weakly proregular. Then there is an isomorphism
\[
\mrm{L}\widehat{\Lambda}_{\b} \mrm{R}\opn{Hom}_A(B,-) \cong  \mrm{R}\opn{Hom}_A(\widehat{B},\mrm{L}\Lambda_{\a}(-))
\]
of functors 
\[
\mrm{D}(\opn{Mod} A) \to \mrm{D}(\opn{Mod} \widehat{B}),
\]
where $\widehat{B} := \Lambda_{\b}(B)$.
\end{cor}
\begin{proof}
Let $\mathbf{a}$ be a finite sequence that generates $\a$, and let $\mathbf{b}$ be its image in $B$. By Theorem \ref{thm:wide-llambda}, given $M \in \mrm{D}(\opn{Mod} A)$, there is a functorial isomorphism
\[
\mrm{L}\widehat{\Lambda}_{\b} \mrm{R}\opn{Hom}_A(B,M) \cong \mrm{R}\opn{Hom}_B( \widehat{B}\otimes_B \opn{Tel}(B;\mathbf{b}), \mrm{R}\opn{Hom}_A(B,M))
\]
so by the derived hom-tensor adjunction, 
\[
\mrm{L}\widehat{\Lambda}_{\b} \mrm{R}\opn{Hom}_A(B,M)  \cong \mrm{R}\opn{Hom}_A(\widehat{B}\otimes_B \opn{Tel}(B;\mathbf{b}),M).
\]
By the base change property of the telescope complex, we have that 
\[
\opn{Tel}(B;\mathbf{b}) \cong B\otimes_A \opn{Tel}(A;\mathbf{a}),
\]
so that
\[
\mrm{L}\widehat{\Lambda}_{\b} \mrm{R}\opn{Hom}_A(B,M)  \cong \mrm{R}\opn{Hom}_A(\widehat{B}\otimes_A \opn{Tel}(A;\mathbf{a}),M).
\]
Hence, using adjunction again, and the fact that $\a$ is weakly proregular, we obtain the result.
\end{proof}

Dually to these two corollaries, we have the following results which will apply to the study of Hochschild homology. We omit the very similar proofs.

\begin{cor}\label{cor:compare-adic-tor}
Let $A$ be a commutative ring, and let $\a\subseteq A$ be a weakly proregular ideal. Given a ring map $\widehat{A} \to B$, and a complex $M \in \mrm{D}(\opn{Mod} B)$,  there is an isomorphism
\[
M \otimes^{\mrm{L}}_A \mrm{R}\Gamma_{\a} (-) \cong M\otimes^{\mrm{L}}_{\widehat{A}} \mrm{R}\widehat{\Gamma}_{\a}(-)
\]
of functors 
\[
\mrm{D}(\opn{Mod} A) \to \mrm{D}(\opn{Mod} B).
\]
\end{cor}

\begin{cor}\label{cor:RGam-ofTen}
Let $A$ be a commutative ring, and let $B$ be a commutative $A$-algebra. Let $\a\subseteq A$ be a weakly proregular ideal, let $\b = \a \cdot B$, and assume that $\b$ is also weakly proregular. Then there is an isomorphism
\[
\mrm{R}\widehat{\Gamma}_{\b} (B\otimes^{\mrm{L}}_A -) \cong \widehat{B}\otimes^{\mrm{L}}_A \mrm{R}\Gamma_{\a}(-))
\]
of functors 
\[
\mrm{D}(\opn{Mod} A) \to \mrm{D}(\opn{Mod} \widehat{B}),
\]
where $\widehat{B} := \Lambda_{\b}(B)$.
\end{cor}

\section{Hochschild cohomology and derived completion}\label{section:main}

We now turn to the main theme of this paper: relations between adic completion (and its derived functor) and Hochschild cohomology. The results of this section rely heavily on the tools developed in the previous sections.

Our first result describes the effect of applying the derived completion functor to the Hochschild cohomology complex in a rather general situation. We will later specialize further to obtain more explicit results.

\begin{thm}\label{thm:LLambda-Of-Hochscild}
Let $\k$ be a commutative ring, let $A$ be a flat noetherian $\k$-algebra, and let $\a\subseteq A$ be an ideal. Assume further that at least one of the following holds:
\begin{enumerate}
\item The ring $\k$ is an absolutely flat ring (e.g., a field).
\item The ring $A/\a$ is essentially of finite type over $\k$.
\item The ideal $I := \a\otimes_{\k} A + A\otimes_{\k} \a \subseteq A\otimes_{\k} A$ is weakly proregular.
\end{enumerate}
Set $\widehat{A} := \Lambda_{\a}(A)$ and $A\widehat{\otimes}_{\k} A := \Lambda_I(A\otimes_{\k} A)$. Then there are isomorphisms
\[
\mrm{L}\widehat{\Lambda}_{\a} \mrm{R}\opn{Hom}_{A\otimes_{\k} A} (A,-) \cong \mrm{R}\opn{Hom}_{A\otimes_{\k} A}(\widehat{A},\mrm{L}\Lambda_I(-)) \cong \mrm{R}\opn{Hom}_{A\widehat{\otimes}_{\k} A}(\widehat{A},\mrm{L}\widehat{\Lambda}_I(-))
\]
of functors $\mrm{D}(\opn{Mod} A\otimes_{\k} A) \to \mrm{D}(\opn{Mod} \widehat{A})$.
\end{thm}
\begin{proof}
If $\k$ is absolutely flat, then $I$ is weakly proregular by Proposition \ref{prop:wpr-over-abflat}, while if $A/\a$ is essentially of finite type over $\k$, $I$ is weakly proregular by Theorem \ref{thm:formally-finite-type-wpr}. Since the image of $I$ in $A$ is equal to $\a$, and since because $A$ is noetherian, $\a$ is weakly proregular, given $M \in \mrm{D}(\opn{Mod} A\otimes_{\k} A)$, by Corollary \ref{cor:LLam-ofHom}, there is a functorial isomorphism
\[
\mrm{L}\widehat{\Lambda}_{\a} \mrm{R}\opn{Hom}_{A\otimes_{\k} A} (A,M) \cong \mrm{R}\opn{Hom}_{A\otimes_{\k} A}(\widehat{A},\mrm{L}\Lambda_I(M)).
\]
Note that, considered as an $A\otimes_{\k} A$-module, the $I$-adic completion of $A$ is equal to $\widehat{A}$. Hence, we may apply Corollary \ref{cor:compare-adic-ext}, and deduce that there is a functorial isomorphism
\[
\mrm{R}\opn{Hom}_{A\otimes_{\k} A}(\widehat{A},\mrm{L}\Lambda_I(M)) \cong 
\mrm{R}\opn{Hom}_{A\widehat{\otimes}_{\k} A}(\widehat{A},\mrm{L}\widehat{\Lambda}_I(M))
\]
in $\mrm{D}(\opn{Mod} \widehat{A})$. This proves the result.
\end{proof}

The next lemma might seem trivial at first glance. However, possible lack of flatness of the completion map, makes it a little more difficult.

\begin{lem}\label{lem:wide-Lambda-Of-Complete}
Let $A$ be a commutative ring, and let $\a\subseteq A$ be a weakly proregular ideal. Assume that $\widehat{A} := \Lambda_{\a}(A)$ is noetherian, and let $M \in \opn{Mod} \widehat{A}$ be an $\a$-adically complete $\widehat{A}$-module. Let 
\[
\opn{Rest}_{\widehat{A}/A}:\mrm{D}(\opn{Mod} \widehat{A}) \to \mrm{D}(\opn{Mod} A)
\]
be the forgetful functor. Then there is a functorial isomorphism
\[
\mrm{L}\Lambda_{\a} ( \opn{Rest}_{\widehat{A}/A}(M) ) \cong \opn{Rest}_{\widehat{A}/A}(M)
\]
in $\mrm{D}(\opn{Mod} A)$,
and a functorial isomorphism
\[
\mrm{L}\widehat{\Lambda}_{\a} ( \opn{Rest}_{\widehat{A}/A}(M) ) \cong M.
\]
in $\mrm{D}(\opn{Mod} \widehat{A})$.
\end{lem}
\begin{proof}
Let $\mathbf{a}$ be a finite sequence that generates $\a$, and let $\widehat{\mathbf{a}}$ be its image in $\widehat{A}$. Since $\widehat{A}$ is noetherian and $M$ is $\a$-adically complete, it follows from \cite[Theorem 1.21]{PSY2} that $M$ is cohomologically $\a\widehat{A}$-adically complete, so there is an isomorphism
\[
M \cong \mrm{R}\opn{Hom}_{\widehat{A}}(\opn{Tel}(\widehat{A};\widehat{\mathbf{a}}),M).
\]
The base change property of the telescope complex implies that $\opn{Tel}(A;\mathbf{a}) \otimes_A \widehat{A} \cong \opn{Tel}(\widehat{A};\widehat{\mathbf{a}})$. Using this fact and the Hom-tensor adjunction, we deduce that
\[
\opn{Rest}_{\widehat{A}/A}(M) \cong \mrm{R}\opn{Hom}_A(\opn{Tel}(A;\mathbf{a}),\opn{Rest}_{\widehat{A}/A}(M)),
\]
so that 
\[
\mrm{L}\Lambda_{\a} ( \opn{Rest}_{\widehat{A}/A}(M) ) \cong \opn{Rest}_{\widehat{A}/A}(M).
\]
Note that
\[
\opn{Rest}_{\widehat{A}/A}( \mrm{L}\widehat{\Lambda}_{\a} (\opn{Rest}_{\widehat{A}/A}(M) ) ) \cong \mrm{L}\Lambda_{\a}(\opn{Rest}_{\widehat{A}/A}(M)) \cong \opn{Rest}_{\widehat{A}/A}(M). 
\]
Hence, the complex of $\widehat{A}$-modules 
\[
\mrm{L}\widehat{\Lambda}_{\a}(\opn{Rest}_{\widehat{A}/A}(M))
\]
is concentrated in degree $0$. Letting $P \to \opn{Rest}_{\widehat{A}/A}(M)$ be a projective resolution, we deduce that the map 
\[
\widehat{\Lambda}_{\a}(P) \to \widehat{\Lambda}_{\a}(\opn{Rest}_{\widehat{A}/A}(M)) \cong M
\]
is an $\widehat{A}$-linear quasi-isomorphism, so the result follows from the fact that 
\[
\mrm{L}\widehat{\Lambda}_{\a} (\opn{Rest}_{\widehat{A}/A}(M)) \cong \widehat{\Lambda}_{\a}(P).
\]
\end{proof}

Using this lemma, and as a first corollary of Theorem \ref{thm:LLambda-Of-Hochscild}, we obtain the next result which reduces the problem of computing the Hochschild cohomology of an adically complete ring to a problem over noetherian rings.

\begin{cor}\label{cor:adic-hh-is-hh}
Let $\k$ be a commutative ring, and let $A$ be a flat noetherian $\k$-algebra. Assume $\a\subseteq A$ is an ideal, such that $A$ is $\a$-adically complete, and such that $A/\a$ is essentially of finite type over $\k$. Let $I = \a\otimes_{\k} A + A\otimes_{\k} \a$, and set $A\widehat{\otimes}_{\k} A := \Lambda_I(A\otimes_{\k} A)$. Then for any $M \in \opn{Mod} A\otimes_{\k} A$ which is $I$-adically complete (for example, any $\a$-adically complete $A$-module, or more particularly, any finitely generated $A$-module), there is a functorial isomorphism
\[
\mrm{R}\opn{Hom}_{A\otimes_{\k} A} (A,M) \cong \mrm{R}\opn{Hom}_{A\widehat{\otimes}_{\k} A}(A,M)
\]
in $\mrm{D}(\opn{Mod} A)$, and the ring $A\widehat{\otimes}_{\k} A$ is noetherian.
\end{cor}
\begin{proof}
That $A\widehat{\otimes}_{\k} A$ is noetherian follows from the fact that $(A\otimes_{\k} A)/I \cong A/\a\otimes_{\k} A/\a$ is noetherian, and $I$ is finitely generated. Since $A= \widehat{A}$, according to Theorem \ref{thm:LLambda-Of-Hochscild}, there is a functorial isomorphism
\[
\mrm{R}\opn{Hom}_{A\otimes_{\k} A}(A,\mrm{L}\Lambda_I(M)) \cong \mrm{R}\opn{Hom}_{A\widehat{\otimes}_{\k} A}(A,\mrm{L}\widehat{\Lambda}_I(M)).
\]
By Lemma \ref{lem:wide-Lambda-Of-Complete}, we have that 
\[
\mrm{L}\Lambda_I(M) \cong M
\]
in $\mrm{D}(\opn{Mod} A\otimes_{\k} A)$, and
\[
\mrm{L}\widehat{\Lambda}_I(M) \cong M
\]
in $\mrm{D}(\opn{Mod} A\widehat{\otimes}_{\k} A)$. Using these two facts, we obtain the functorial isomorphism
\[
\mrm{R}\opn{Hom}_{A\otimes_{\k} A} (A,M) \cong \mrm{R}\opn{Hom}_{A\widehat{\otimes}_{\k} A}(A,M).
\]
\end{proof}

\begin{rem}
In \cite[Section 3, Page 113]{BF}, the authors defined the analytic Hochschild cohomology and discussed its relation to ordinary Hochschild cohomology. The setup there is as follows: $(A,\mfrak{m}_A)$ and $(B,\mfrak{m}_B)$ are two complete noetherian local rings, and there is a flat local map $A\to B$ such that the induced map $A/\mfrak{m}_A \to B/\mfrak{m}_B$ is an isomorphism. In this situation, the authors defined the analytic bar resolution $\hat{\mathcal{B}}$ by replacing $B^{\otimes^n_A}$ with its completion  with respect to the maximal ideal $\ker(B^{\otimes^n_A} \to B/\mfrak{m}_B)$ in the ordinary bar resolution. Using this complex, the authors defined the analytic Hochschild cohomology by 
\[
\widehat{\mrm{HH}}^n(B/A,M) := H^n \opn{Hom}_{B\widehat{\otimes}_A B}(\hat{\mathcal{B}},M),
\]
observed that there is a canonical map 
\[
\widehat{\mrm{HH}}^n(B/A,M) \to\mrm{HH}^n(B/A,M),
\]
and asked if these modules are isomorphic. Using Corollary \ref{cor:adic-hh-is-hh} and the results of \cite{BF}, we may now obtain a positive answer to this question:
\end{rem}

\begin{cor}\label{cor:Answer-To_BF}
Let $(A,\mfrak{m}_A) \to (B,\mfrak{m}_B)$ be a flat local homomorphism of complete noetherian local rings, such that the induced map of residue fields $A/\mfrak{m}_A \to B/\mfrak{m}_B$ is an isomorphism. Then for any $\mfrak{m}_B$-adically complete $B$-module $M$ (in particular, for any finitely generated $B$-module $M$), there is a natural $B$-module isomorphism
\[
\widehat{\mrm{HH}}^n(B/A,M) \cong \mrm{HH}^n(B/A,M).
\]
\end{cor}
\begin{proof}
Note that $M$ is $\ker(B\otimes_A B \to B/\mfrak{m}_B)$-adically complete, so by \cite[Proposition 3.1]{BF}, the natural map
\[
\mrm{HH}^n(B/A,M) \to \opn{Ext}^n_{B\otimes_A B}(B,M)
\]
is an isomorphism. Similarly, by \cite[Proposition 3.2]{BF}, the natural map
\[
\widehat{\mrm{HH}}^n(B/A,M) \to \opn{Ext}^n_{B\widehat{\otimes}_A B}(B,M)
\]
is an isomorphism. Finally, by Corollary \ref{cor:adic-hh-is-hh}, there is a natural isomorphism
\[
\opn{Ext}^n_{B\otimes_A B}(B,M) \cong  \opn{Ext}^n_{B\widehat{\otimes}_A B}(B,M),
\]
so the result follows.
\end{proof}

Our next goal is to give the complex appearing in Theorem \ref{thm:LLambda-Of-Hochscild} a description in terms of the Hochschild complex of the ring $\widehat{A}$. First we need a lemma.

\begin{lem}\label{lem:completion-of-tensor-of-completion}
Let $\k$ be a commutative ring,  let $A$ be a flat noetherian $\k$-algebra, and let $\a\subseteq A$ be an ideal. Let 
\[
I :=\a\otimes_{\k} A + A\otimes_{\k} \a \subseteq A\otimes_{\k} A,
\]
let $\widehat{\a} := \a\cdot \Lambda_{\a}(A)$, and let 
\[
J = \widehat{\a} \otimes_{\k} \widehat{A} + \widehat{A}\otimes_{\k} \widehat{\a} = I \cdot (\widehat{A}\otimes_{\k} \widehat{A}) \subseteq \widehat{A}\otimes_{\k} \widehat{A}.
\]
If $\tau_A : A \to \widehat{A}$ is the completion map, then the ring map
\[
\Lambda_I(\tau_A \otimes_k \tau_A) : \Lambda_I(A\otimes_{\k} A) \to \Lambda_J(\widehat{A}\otimes_{\k} \widehat{A})
\]
is an isomorphism.
\end{lem}
\begin{proof}
Notice that for each $n$ we have that
\[
I^n = \sum_{i=0}^n \a^i \otimes \a^{n-i},
\]
where we have set $\a^0 = A$. Hence, we have that 
\[
\a^{2n} \otimes_{\k} A + A\otimes_{\k} \a^{2n} \subseteq I^{2n} \subseteq \a^n\otimes_{\k} A + A\otimes_{\k} \a^n.
\]
It follows that the two sequences of ideals $\{I^n\}$ and $\{\a^n\otimes_{\k} A + A\otimes_{\k} \a^n\}$ are cofinal in each other, so that
\[
\varprojlim (A\otimes_{\k} A)/I^n \cong \varprojlim (A\otimes_{\k} A)/(\a^n\otimes_{\k} A + A\otimes_{\k} \a^n) \cong \varprojlim A/{\a}^n\otimes_{\k} A/{\a}^n.
\]
Since $A$ is noetherian, $\widehat{A}$ is flat over $A$, so it is also flat over $\k$. Hence, in the exact same manner,
\[
\varprojlim (\widehat{A}\otimes_{\k} \widehat{A})/J^n \cong \varprojlim \widehat{A}/{(\widehat{\a}})^n\otimes_{\k} A/{(\widehat{\a})}^n.
\]
Hence, the result follows from the fact that $\widehat{A} / (\widehat{\a})^n \cong A/{\a}^n$, and from the observation that the maps $\Lambda_I(\tau_A \otimes_k \tau_A)$
and 
\[
\varprojlim ((\tau_A \otimes_k \tau_A)/(\a^n\otimes_{\k} A + A\otimes_{\k} \a^n) )
\]
are equal.
\end{proof}

\begin{rem}
In the notations of the above lemma, note that composing the completion map 
\[
\widehat{A}\otimes_{\k} \widehat{A} \to \Lambda_J(\widehat{A}\otimes_{\k} \widehat{A})
\]
with the isomorphism
\[
\Lambda_J(\widehat{A}\otimes_{\k} \widehat{A}) \to \Lambda_I(A\otimes_{\k} A)
\]
we obtain a ring map
\[
\widehat{A}\otimes_{\k} \widehat{A} \to \Lambda_I(A\otimes_{\k} A).
\]
Using this map, given $M \in \opn{Mod} (A\otimes_{\k} A)$, we will be able to regard 
\[
\widehat{\Lambda}_I(M) \in \opn{Mod}(\Lambda_I(A\otimes_{\k} A))
\]
as an $\widehat{A}\otimes_{\k} \widehat{A}$-module which is $J$-adically complete.
\end{rem}

We may now refine Theorem \ref{thm:LLambda-Of-Hochscild}, and present the derived completion of the Hochschild cohomology complex as a Hochschild cohomology complex over the completion:

\begin{thm}\label{thm:der-com-of-ext}
Let $\k$ be a commutative ring, and let $A$ be a flat noetherian $\k$-algebra such that $A\otimes_{\k} A$ is noetherian. Let $\a\subseteq A$ be an ideal, and let $M$ be a finitely generated $A\otimes_{\k} A$-module. Set 
\[
I := \a\otimes_{\k} A + A\otimes_{\k} \a \subseteq A\otimes_{\k} A.
\]
Then there is a functorial isomorphism
\[
\mrm{L}\widehat{\Lambda}_{\a} ( \mrm{R}\opn{Hom}_{A\otimes_{\k} A}(A,M) ) \cong \mrm{R}\opn{Hom}_{\widehat{A}\otimes_{\k} \widehat{A} }(\widehat{A}, \widehat{M}),
\]
in $\mrm{D}(\opn{Mod} \widehat{A})$, where $\widehat{A} := \widehat{\Lambda}_{\a}(A)$ and $\widehat{M} := \widehat{\Lambda}_I(M)$.
\end{thm}
\begin{proof}
Set $A\widehat{\otimes}_{\k} A := \Lambda_I(A\otimes_{\k} A)$.  By assumption, $A\otimes_{\k} A$ is noetherian, so $I$ is weakly proregular.
According to Theorem \ref{thm:LLambda-Of-Hochscild}, there is a functorial isomorphism
\begin{equation}\label{eqn:der-comp1}
\mrm{L}\widehat{\Lambda}_{\a} ( \mrm{R}\opn{Hom}_{A\otimes_{\k} A}(A,M) ) \cong 
\mrm{R}\opn{Hom}_{A\widehat{\otimes}_{\k} A}(\widehat{A},\mrm{L}\widehat{\Lambda}_I(M)). 
\end{equation}
Because $A\otimes_{\k} A$ is noetherian, and $M$ is a finitely generated module, it follows that 
\[
\mrm{L}\widehat{\Lambda}_I(M) \cong \widehat{\Lambda}_I(M) = \widehat{M}
\]
in $\mrm{D}(\opn{Mod} A\widehat{\otimes}_{\k} A)$. Hence, by Lemma \ref{lem:completion-of-tensor-of-completion}, there is a functorial isomorphism
\begin{equation}\label{eqn:der-comp2}
\mrm{R}\opn{Hom}_{A\widehat{\otimes}_{\k} A}(\widehat{A},\widehat{M}) \cong
\mrm{R}\opn{Hom}_{\widehat{A} \widehat{\otimes}_{\k} \widehat{A}}(\widehat{A},\widehat{M}) 
\end{equation}
in $\mrm{D}(\opn{Mod} \widehat{A})$. Let $\widehat{\a} := \a\cdot \widehat{A} \subseteq \widehat{A}$.
Since $A$ is noetherian, the map $A\to\widehat{A}$ is flat, so the map $A\otimes_{\k} A \to \widehat{A}\otimes_{\k} \widehat{A}$ is also flat. Hence, by Remark \ref{remark-wpr-flat}, the ideal $J := I\cdot (\widehat{A}\otimes_{\k} \widehat{A}) = \widehat{\a}\otimes_{\k} \widehat{A} + \widehat{A}\otimes_{\k} \widehat{\a}$ is weakly proregular. 

Because $A\otimes_{\k} A$ is noetherian, its completion $\Lambda_I(A\otimes_{\k} A)$ is also noetherian, so again by Lemma 4.7, the ring $\Lambda_J(\widehat{A}\otimes_{\k} \widehat{A})$ is noetherian. Hence, by Lemma \ref{lem:wide-Lambda-Of-Complete}, we have that $\mrm{L}\Lambda_J(\widehat{M}) \cong \mrm{L}\widehat{\Lambda}_J(\widehat{M}) \cong \widehat{M}$. Applying Theorem \ref{thm:LLambda-Of-Hochscild} to the ring $\widehat{A}$, we have a functorial isomorphism
\[
\mrm{R}\opn{Hom}_{\widehat{A}\otimes_{\k} \widehat{A}}(\widehat{A}, \mrm{L}\Lambda_J (M)) \cong 
\mrm{R}\opn{Hom}_{\widehat{A}\widehat{\otimes}_{\k} \widehat{A}}(\widehat{A}, \mrm{L}\widehat{\Lambda}_J (M)).
\]
Composing this isomorphism with the isomorphisms of equations (\ref{eqn:der-comp1}) and (\ref{eqn:der-comp2}), we obtain the result.
\end{proof}

Our final goal in this section is to show that when taking cohomology in the above theorem, derived completion may be replaced with ordinary completion. The next simple lemma is needed for the proof of Proposition \ref{prop:completion-of-fg}.

\begin{lem}\label{lem:hom-from-free}
Let $A$ be a noetherian ring, let $P$ be a bounded complex of free $A$-modules, and let $M$ be a finitely generated $A$-module. Then the canonical map
\[
\opn{Hom}_A(P,A) \otimes_A M \to \opn{Hom}_A(P,M)
\]
is an isomorphism of complexes.
\end{lem}
\begin{proof}
It is enough to show this in the case where $P$ is a single free $A$-module. In that case, note that $\opn{Hom}_A(P,A)$ is a direct product of copies of $A$, and since $A$ is noetherian, this is a flat $A$-module. Thus, both of the functors $\opn{Hom}_A(P,A)\otimes_A -$ and $\opn{Hom}_A(P,-)$ are exact, and if $M$ is a finitely generated free $A$-module, then the canonical map $\opn{Hom}_A(P,A) \otimes_A M \to \opn{Hom}_A(P,M)$ is obviously an isomorphism. Hence, the result of the lemma follows from the standard finite presentation trick.
\end{proof}

In case $M$ is bounded above, the next Proposition is \cite[Proposition 2.7]{Fr}. We however wish to apply this result to the Hochschild complex which is bounded below, so we give a proof that works for complexes without any boundedness condition.

\begin{prop}\label{prop:completion-of-fg}
Let $A$ be a noetherian ring, and let $\a\subseteq A$ be an ideal. Then there is an isomorphism
\[
\mrm{L}\Lambda_{\a} (M) \cong \widehat{A}\otimes_A M
\]
of functors $\mrm{D}_{\mrm{f}} (\opn{Mod} A) \to \mrm{D}(\opn{Mod} A)$.
\end{prop}
\begin{proof}
There is a sequence of morphisms of functors
\[
\widehat{A}\otimes_A M \cong \opn{Hom}_A(\opn{Tel}(A;\mathbf{a}), A) \otimes_A M \to \opn{Hom}_A(\opn{Tel}(A;\mathbf{a}), M) \cong \mrm{L}\Lambda_{\a} (M).
\]
Because $A$ is assumed to be noetherian, $\widehat{A}$ is flat over $A$. Hence, both $\mrm{L}\Lambda_{\a}(-)$ and $-\otimes_A \widehat{A}$ are functors of finite cohomological dimension. Thus, by \cite[Proposition I.7.1]{RD}, it is enough to show that the above morphism is an isomorphism in the case where $M$ is a finitely generated $A$-module, and this follows from Lemma \ref{lem:hom-from-free}.
\end{proof}

Here is result promised in the title of this paper:

\begin{thm}\label{thm:main}
Let $\k$ be a commutative ring, and let $A$ be a flat noetherian $\k$-algebra such that $A\otimes_{\k} A$ is noetherian. Let $\a\subseteq A$ be an ideal, and let $M$ be a finitely generated $A\otimes_{\k} A$-module.
Then for any $n \in \mathbb{N}$, there is a functorial isomorphism
\[
\Lambda_{\a} \left(\opn{Ext}^n_{A\otimes_{\k} A}(A,M)\right) \cong \opn{Ext}^n_{\widehat{A}\otimes_{\k} \widehat{A}}(\widehat{A},\widehat{M})
\]
If moreover, either
\begin{enumerate}
\item $\k$ is a field, or 
\item $A$ is projective over $\k$, $\a$ is a maximal ideal, and $M$ is a finitely generated $A$-module,
\end{enumerate}
there is also a functorial isomorphism
\[
\Lambda_{\a} \left( \mrm{HH}^n(A/\k,M )  \right) \cong \mrm{HH}^n(\widehat{A}/\k,\widehat{M}).
\]
\end{thm}
\begin{proof}
The assumptions of the theorem ensure that 
\[
\mrm{R}\opn{Hom}_{A\otimes_{\k} A}(A,M) \in \mrm{D}_{\mrm{f}}(\opn{Mod} A),
\]
so since $A$ is noetherian, we have a functorial isomorphism
\[
\Lambda_{\a} \left(\opn{Ext}^n_{A\otimes_{\k} A}(A,M)\right) \cong \widehat{A} \otimes_A H^n(\mrm{R}\opn{Hom}_{A\otimes_{\k} A}(A,M)).
\]
Flatness of $\widehat{A}$ over $A$ implies (for example by \cite[Corollary 2.12]{PSY1}) that there is a natural isomorphism
\[
 \widehat{A} \otimes_A H^n(\mrm{R}\opn{Hom}_{A\otimes_{\k} A}(A,M)) \cong H^n(\widehat{A} \otimes_A  \mrm{R}\opn{Hom}_{A\otimes_{\k} A}(A,M) ).
\]
Hence, by Proposition \ref{prop:completion-of-fg}, it is enough to compute the $n$-th cohomology of the complex
\[
\mrm{L}\Lambda_{\a} (\mrm{R}\opn{Hom}_{A\otimes_{\k} A}(A,M)).
\]
Letting 
\[
\opn{Rest}_{\widehat{A}/A}:\mrm{D}(\opn{Mod} \widehat{A}) \to \mrm{D}(\opn{Mod} A)
\]
be the forgetful functor, the above complex is equal to
\[
\opn{Rest}_{\widehat{A}/A} \circ \mrm{L}\widehat{\Lambda}_{\a}(\mrm{R}\opn{Hom}_{A\otimes_{\k} A}(A,M)).
\]
By Theorem \ref{thm:der-com-of-ext}, there is a functorial isomorphism
\[
\mrm{L}\widehat{\Lambda}_{\a}\mrm{R}\opn{Hom}_{A\otimes_{\k} A}(A,M) \cong 
\mrm{R}\opn{Hom}_{\widehat{A}\otimes_{\k} \widehat{A}}(\widehat{A},\widehat{M})
\]
in $\mrm{D}(\opn{Mod} \widehat{A})$, so applying the forgetful functor we obtain an $A$-linear natural isomorphism
\[
\Lambda_{\a} (\opn{Ext}^n_{A\otimes_{\k} A}(A,M))  \cong \opn{Ext}^n_{\widehat{A}\otimes_{\k} \widehat{A}}(\widehat{A},\widehat{M}).
\]
(Actually, any $A$-linear map between $\widehat{A}$-modules is automatically $\widehat{A}$-linear, so this isomorphism is even an isomorphism of $\widehat{A}$-modules). This establishes the first claim of the theorem. If $\k$ is a field then the second claim obviously follows from the first one. Assume now that $A$ is projective over $\k$, that $\a$ is a maximal ideal, and that $M$ is a finitely generated $A$-module. Let $\phi$ be the composition of the diagonal map $\widehat{A}\otimes_{\k} \widehat{A} \to \widehat{A}$ with the map $\widehat{A} \to \widehat{A}/\a\widehat{A}$. Then $m = \ker(\phi) \subseteq \widehat{A}\otimes_{\k} \widehat{A}$ is a maximal ideal, and the image of $m$ in $\widehat{A}$ is equal to $\a\widehat{A}$. Hence $\widehat{M}$ is $m$-adically complete, so by \cite[Proposition 3.1]{BF}, the canonical map
\[
\mrm{HH}^n(\widehat{A}/\k,\widehat{M}) \to \opn{Ext}^n_{\widehat{A}\otimes_{\k} \widehat{A}}(\widehat{A},\widehat{M})
\]
is an isomorphism. This proves the second claim.
\end{proof}

The above result allows us to compute the Hochschild cohomology of power series rings, which is new as far as we know.

\begin{exa}
Let $\k$ be a noetherian ring, let $A = \k[x_1,\dots,x_n]$, let $\a=(x_1,\dots,x_n)$, and let $M=A$. Note that $A$ is projective over $\k$, and that $\Lambda_{\a}(A) = \k[[x_1,\dots,x_n]]$. Hence, by the above theorem, we have that
\[
\opn{Ext}^i_{\widehat{A}\otimes_{\k} \widehat{A}}(\widehat{A},\widehat{A})
\cong \Lambda_{\a} ( \mrm{HH}^i(\k[x_1,\dots,x_n]/\k,\k[x_1,\dots,x_n] ).
\]
By the Hochschild-Kostant-Rosenberg theorem, the right hand side is equal to
\[
\Lambda_{\a} \wedge^i ( \k[x_1,\dots,x_n]^n ),
\]
so we have an isomorphism
\[
\opn{Ext}^i_{\widehat{A}\otimes_{\k} \widehat{A}}(\widehat{A},\widehat{A}) \cong \wedge^i ( \k[[x_1,\dots,x_n]]^n ).
\]
If moreover $\k$ is a field, we obtain that
\[
\mrm{HH}^i(\k[[x_1,\dots,x_n]]/\k,\k[[x_1,\dots,x_n]]) \cong \wedge^i ( \k[[x_1,\dots,x_n]]^n ).
\]
We remark that one can also use Corollary \ref{cor:adic-hh-is-hh} to make this calculation. 
\end{exa}

\begin{exa}
Let $A$ be a noetherian ring, and let $\a\subseteq A$ be an ideal. Then by Theorem \ref{thm:der-com-of-ext}, we have that
\[
\mrm{R}\opn{Hom}_{ \widehat{A} \otimes_A \widehat{A} } (\widehat{A}, \widehat{A} ) \cong \mrm{L}\widehat{\Lambda}_{\a} \mrm{R}\opn{Hom}_A(A,A) = \widehat{A}.
\]
Assume now that $\a$ is a maximal ideal. Then by Theorem \ref{thm:main}  we have that
\[
\mrm{HH}^n(\widehat{A}/A,\widehat{A}) = 0
\]
for all $n \ne 0$, and $\mrm{HH}^0(\widehat{A}/A,\widehat{A}) = \widehat{A}$. Specializing to the case where $A = \mathbb{Z}$, and $\a = (p)$ for some prime number $p$, we have computed the absolute Hochschild cohomology of the ring of p-adic integers $\mathbb{Z}_p$.
\end{exa}

\begin{rem}
The above examples are both particular cases of an adic Hochschild-Kostant-Rosenberg theorem which applies for any 	formally smooth adic algebra $A$ such that $A/\a$ is essentially of finite type over $\k$. 
This can be shown using Corollary \ref{cor:adic-hh-is-hh} and by studying the local structure of the completed diagonal $\ker(A\widehat{\otimes}_{\k} A \to A)$. A full proof of this will appear elsewhere.
\end{rem}

\section{Hochschild homology and derived torsion}\label{section:hhomology}

In this short and final section we discuss relations between Hochschild homology and the derived torsion functor.

\begin{thm}\label{thm:RGamma-Of-HochscildHomology}
Let $\k$ be a commutative ring, let $A$ be a flat noetherian $\k$-algebra, and let $\a\subseteq A$ be an ideal. Assume further that at least one of the following holds:
\begin{enumerate}
\item The ring $\k$ is an absolutely flat ring (e.g., a field).
\item $A/\a$ is essentially of finite type over $\k$.
\item The ideal $I:=\a\otimes_{\k} A + A\otimes_{\k} \a \subseteq A\otimes_{\k} A$ is weakly proregular.
\end{enumerate}
Set $\widehat{A} := \Lambda_{\a}(A)$, $A\widehat{\otimes}_{\k} A := \Lambda_I(A\otimes_{\k} A)$. Then there are isomorphisms
\[
\mrm{R}\widehat{\Gamma}_{\a} ( A\otimes^{\mrm{L}}_{A\otimes_{\k} A} -) \cong
\widehat{A} \otimes^{\mrm{L}}_{A\otimes_{\k} A} \mrm{R}\Gamma_I(-) \cong
\widehat{A} \otimes^{\mrm{L}}_{A\widehat{\otimes}_{\k} A} \mrm{R}\widehat{\Gamma}_I(-)
\]
of functors $\mrm{D}(\opn{Mod} A\otimes_{\k} A) \to \mrm{D}(\opn{Mod} \widehat{A})$.
\end{thm}
\begin{proof}
As in the proof of Theorem \ref{thm:LLambda-Of-Hochscild}, the first two conditions imply the third one, so we may assume $I$ is weakly proregular. The first isomorphism then follows from Corollary \ref{cor:RGam-ofTen}, while the second isomorphism follows from Corollary \ref{cor:compare-adic-tor}.
\end{proof}

\begin{rem}
In view of Theorem \ref{thm:main}, it is natural to ask if Hochschild homology also commutes with adic-completion. Here, the answer is false, even in simple situations. Indeed, let $\k$ be a field of characteristic $0$, let $A = \k[x]$, and let $\a=(x)$. Then 
\[
\mrm{HH}_1(A/\k,A) \cong A,
\]
so that $\Lambda_{\a} (\mrm{HH}_1(A/\k,A)) \cong \k[[x]]$. On the other hand, $\mrm{HH}_1(\widehat{A}/\k,\widehat{A}) \cong \Omega^1_{\k[[x]]/\k}$ is an infinitely generated $\k[[x]]$-module.
\end{rem}

\begin{rem}
As alternative to the badly behaved Hochschild homology of commutative adic algebras, Hubl (\cite{Hu}) developed a theory of adic Hochschild homology, by studying the cohomologies of the functor
\[
A\otimes^{\mrm{L}}_{A\widehat{\otimes}_{\k} A} -.
\]
As our Corollary \ref{cor:adic-hh-is-hh} shows, in the case of Hochschild cohomology, usual Hochschild cohomology coincides with adic Hochschild cohomology, but by the previous remark we see that for Hochschild homology this is not the case.
\end{rem}

\textbf{Acknowledgments.}
The author would like to thank Amnon Yekutieli for some helpful suggestions.
The author is grateful for the anonymous referees for their helpful comments.

\end{document}